\documentclass[12pt]{amsart}
\usepackage{amssymb,amsthm,amsmath,epsfig,latexsym}
\usepackage{calc}
\usepackage{times}

\usepackage{ifthen}
\newcommand{\showcomments}{yes}
\newsavebox{\commentbox}
\newenvironment{comment}%
{\ifthenelse{\equal{\showcomments}{yes}}%
{\footnotemark
    \begin{lrbox}{\commentbox}
    \begin{minipage}[t]{1.25in}\raggedright\sffamily\upshape\tiny
    \footnotemark[\arabic{footnote}]}
{\begin{lrbox}{\commentbox}}}%
{\ifthenelse{\equal{\showcomments}{yes}}%
{\end{minipage}\end{lrbox}\marginpar{\usebox{\commentbox}}}
{\end{lrbox}}}

\begin{document}

\newcommand{\mmbox}[1]{\mbox{${#1}$}}
\newcommand{\proj}[1]{\mmbox{{\mathbb P}^{#1}}}
\newcommand{\Cr}{C^r(\Delta)}
\newcommand{\CR}{C^r(\hat\Delta)}
\newcommand{\affine}[1]{\mmbox{{\mathbb A}^{#1}}}
\newcommand{\Ann}[1]{\mmbox{{\rm Ann}({#1})}}
\newcommand{\caps}[3]{\mmbox{{#1}_{#2} \cap \ldots \cap {#1}_{#3}}}
\newcommand{\N}{{\mathbb N}}
\newcommand{\Z}{{\mathbb Z}}
\newcommand{\R}{{\mathbb R}}
\newcommand{\Tor}{\mathop{\rm Tor}\nolimits}
\newcommand{\Ext}{\mathop{\rm Ext}\nolimits}
\newcommand{\Hom}{\mathop{\rm Hom}\nolimits}
\newcommand{\im}{\mathop{\rm Im}\nolimits}
\newcommand{\rank}{\mathop{\rm rank}\nolimits}
\newcommand{\supp}{\mathop{\rm supp}\nolimits}
\newcommand{\arrow}[1]{\stackrel{#1}{\longrightarrow}}
\newcommand{\CB}{Cayley-Bacharach}
\newcommand{\coker}{\mathop{\rm coker}\nolimits}
\sloppy
\newtheorem{defn0}{Definition}[section]
\newtheorem{prop0}[defn0]{Proposition}
\newtheorem{conj0}[defn0]{Conjecture}
\newtheorem{thm0}[defn0]{Theorem}
\newtheorem{lem0}[defn0]{Lemma}
\newtheorem{corollary0}[defn0]{Corollary}
\newtheorem{example0}[defn0]{Example}
\newtheorem{note0}[defn0]{Note}
\newtheorem{question0}[defn0]{Question}

\newenvironment{defn}{\begin{defn0}}{\end{defn0}}
\newenvironment{prop}{\begin{prop0}}{\end{prop0}}
\newenvironment{conj}{\begin{conj0}}{\end{conj0}}
\newenvironment{question}{\begin{question0}}{\end{question0}}
\newenvironment{thm}{\begin{thm0}}{\end{thm0}}
\newenvironment{lem}{\begin{lem0}}{\end{lem0}}
\newenvironment{cor}{\begin{corollary0}}{\end{corollary0}}
\newenvironment{exm}{\begin{example0}\rm}{\end{example0}}
\newenvironment{note}{\begin{note0}\rm}{\end{note0}}

\newcommand{\defref}[1]{Definition~\ref{#1}}
\newcommand{\propref}[1]{Proposition~\ref{#1}}
\newcommand{\thmref}[1]{Theorem~\ref{#1}}
\newcommand{\lemref}[1]{Lemma~\ref{#1}}
\newcommand{\noteref}[1]{Note~\ref{#1}}
\newcommand{\corref}[1]{Corollary~\ref{#1}}
\newcommand{\exref}[1]{Example~\ref{#1}}
\newcommand{\secref}[1]{Section~\ref{#1}}

\newcommand{\std}{Gr\"{o}bner}
\newcommand{\jq}{J_{Q}}



\title {From Spline Approximation to Roth's Equation and Schur Functors}

\author{J\'{a}n Min\'{a}\v{c}}

\author{\c Stefan O. Toh\v aneanu}

\thanks{J\'{a}n Min\'{a}\v{c}: Department of Mathematics, Western University, London, Ontario N6A 5B7, Canada. Email: minac@uwo.ca, Phone: 519-661-2111, Ext: 86519, Fax: 519-661-3610.}
\thanks{\c Stefan O. Toh\v aneanu (corresponding author): Department of Mathematics, Western University, London, Ontario N6A 5B7, Canada. Email: stohanea@uwo.ca, Phone: 519-661-2111, Ext: 86528, Fax: 519-661-3610.}

\begin{abstract}
\noindent Alfeld and Schumaker provide a formula for the dimension of the space of piecewise polynomial functions, called splines, of degree $d$ and smoothness $r$ on a generic triangulation of a planar simplicial complex $\Delta$, for $d \geq 3r+1$. Schenck and Stiller conjectured that this formula actually holds for all $d \geq 2r+1$. Up to this moment there was not known a single example where one could show that the bound $d\geq 2r +1$ is sharp. However, in 2005, a possible such example was constructed to show that this bound is the best possible (i.e., the Alfeld-Schumaker formula does not hold if $d=2r$), except that the proof that this formula actually works if $d\geq 2r+1$ has been a challenge until now when we finally show it to be true. The interesting subtle connections with representation theory, matrix theory and commutative and homological algebra seem to explain why this example presented such a challenge. Thus in this paper we present the first example when it is known that the bound $d\geq 2r+1$ is sharp for asserting the validity of the Alfeld-Schumaker formula.
\end{abstract}
\maketitle


\section{Introduction and background}

Let $\Delta$ be a connected finite simplicial complex whose geometric realization $|\Delta|$ is a topological disk in $\mathbb R^2$. Let $r\geq0$ be an integer. The space of splines of smoothness $r$ and degree $d$ is the $\mathbb R-$ vector space $$C_d^r(\Delta)=\{F:|\Delta|\longrightarrow\mathbb R: F|_{\sigma}=\mbox{ polynomial of degree }\leq d,$$ $$\forall \sigma\in\Delta_2,\mbox{ and }F\in C^r\}.$$ A very nice accessible introduction to Polynomial Splines is \cite[Chapter 8]{clo}.

One of the major questions in spline approximation is to find the dimension of this vector space; even when $d=3$ and $r=1$, this dimension is not known for arbitrary triangulations. If $d\geq 3r+1$, for almost all triangulations, Alfeld and Schumaker (\cite{as}) give a beautiful, yet complicated formula for this dimension in terms of combinatorial and local geometric data (data depending only on local geometry at the interior vertices of $\Delta$):

$$\dim C_d^r(\Delta)={{d+2}\choose{2}}+{{d-r+1}\choose{2}}f_1^0-\left({{d+2}\choose{2}}-{{r+2}\choose{2}}\right)f_0^0 +\sigma,$$ where $f_1^0$ is the number of interior edges of $\Delta$, $f_0^0$ is the number of interior vertices of $\Delta$, and $\sigma=\sum\sigma_i$, where $\sigma_i=\sum_{j\geq 1}\max\{(r+1+j(1-n(v_i))),0\}$, and $n(v_i)$ is the number of distinct slopes at the interior vertex $v_i$. For further reference this formula will be denoted by $L(\Delta,r,d)$.

In \cite{ss1}, by showing that a certain zeroth local cohomology is zero, Schenck and Stillman prove that if $\Delta$ has only pseudoedges (such a triangulation is called {\em quasi-cross-cut}), the Alfeld-Schumaker formula is true for any $d$. With different methods, in \cite{swy} the same is true, but for the more general case when instead of a triangulation, one has a partition. Lemma 2.5 in \cite{ss1} also says that if $\Delta$ has at least one non-pseudoedge, the local cohomology module considered is not zero. In fact, Schenck and Stiller conjectured that for any $\Delta$, this local cohomology module vanishes in degree $d\geq 2r+1$.

\cite{ss} considered a simplicial complex $\Delta_S$ with exactly one non-pseudoedge, and in \cite{to} it was shown that for this particular example, the above conjecture is tight: for any $r\geq 1$, $\dim C_{2r}^r(\Delta_S)\neq L(\Delta_S,r,2r)$. The present notes are a followup of \cite{to}. We show that for the same simplicial complex $\Delta_S$, $$\dim C_{d}^r(\Delta_S)= L(\Delta_{S},r,d),\mbox{ for any }r \mbox{ and }d\geq 2r+1,$$ and therefore the Schenck-Stiller conjecture is true for this first non-trivial triangulation $\Delta_S$\footnote{See Subsection 1.2.}.

The confirmation of the Schenck-Stiller conjecture in this case proved to be surprisingly challenging, and in fact, it took several years to establish the main result of this paper. The reason for this difficulty seems to lie in rather deep connections with representation theory, matrix theory and commutative algebra.

The proof is subtle in several places, but we put considerable effort into making our exposition clear and readable also for the non-specialists. In the first part we provide the solution of the main problem that relies on classical concepts in commutative algebra (e.g., regular sequences, monomial order, etc.). In the second part of the paper we investigate the connections with Schur functors, Roth's equation in matrix theory and lower-upper triangular matrix decompositions.

\subsection{Homological approach to spline approximation}

Following \cite{br91}, consider $\mathbb R^2$ embedded in $\mathbb R^3$, and let $\hat{\Delta}$ be the cone of $\Delta$ with its origin in $\mathbb R^3$ and let $R=\mathbb R[x,y,z]$ be the ring of polynomials in variables $x,y,z$ with coefficients in $\mathbb R$. Since we consider the cone of $\Delta$, from now on if $e$ is an edge of $\Delta$ we will think of $\ell_e$ to be the homogenized equation of the equation of the line in $\mathbb R^2$ where $e$ is placed; also, abusing the terminology a bit, we are going to say that the linear form $\ell_e$ defines the edge $e$.

Consider the finitely generated graded $R-$module: $$C^r(\hat{\Delta})=\{F:|\hat{\Delta}|\longrightarrow\mathbb R: F|_{\hat{\sigma}}\in R, \forall \sigma\in\Delta_2, \mbox{ and }F\in C^r\}.$$ Then $$\dim_{\mathbb R}C_d^r(\Delta)=\dim_{\mathbb R}C^r(\hat{\Delta})_d,$$ the dimension of the degree $d$ piece of graded module $C^r(\hat{\Delta})$. So, by taking $\hat{\Delta}$ the cone of $\Delta$, we homogenized our polynomials and therefore the problem is translated into a homological algebra problem: to find the Hilbert function of a graded module \footnote{ If $M=\oplus_{d\in\mathbb Z}M_d$ is a finitely generated graded $R-$module, the Hilbert function of $M$ in degree $d$ is by definition $HF(M,d)=\dim_{\mathbb R}M_d$. See, for example, \cite{s} for more background on Hilbert functions.}.

Let $F\in C^r(\hat{\Delta})_d$. Piecewise, on each triangle $T_i$ of $\Delta$, $F$ is defined by a homogeneous polynomial of degree $d$: $F_i\in R_d$. For $F$ to be a $C^r-$function, since polynomials are $C^{\infty}-$functions, whenever we have two triangles $T_i$ and $T_j$, with a common (interior) edge of equation $\ell_{ij}=0$, then $$F_i-F_j\in\langle \ell_{ij}^{r+1}\rangle.$$ For example, if $r=0$ (i.e., $F$ is continuous), one must have $F_i(P)=F_j(P),$ for all $P\in V(\ell_{ij})$. But this means exactly that $F_i-F_j\in \langle \ell_{ij}\rangle.$

With the above idea in mind, Billera and Rose (\cite{br2}) place $C^r(\hat{\Delta})$ in the following exact complex of graded $R-$modules: $$0\rightarrow C^r(\hat{\Delta})\rightarrow R^{f_2}\oplus R^{f_1^0}(-r-1)\stackrel{\phi}\rightarrow R^{f_1^0}\rightarrow N\rightarrow 0,$$ where $N=coker(\phi)$, and $$\phi=\left(\begin{array}{ccccc}
&&\ell_{e_1}^{r+1}&&\\
\partial_2&|&&\ddots&\\
&&&&\ell_{e_{f_1^0}}^{r+1}
\end{array}\right).$$ $\partial_2$ is the simplicial reduced boundary map $R^{f_2}\rightarrow R^{f_1^0}$, where $f_2$ is the number of triangles of $\Delta$ and $\ell_{e_i}$ is the linear form that defines the interior edge $e_i$.

Using the properties of the Hilbert function, we obtain $$\dim_{\mathbb R}C^r(\hat{\Delta})_d={{d+2}\choose{2}}f_2+{{d-r+1}\choose{2}}f_1^0-{{d+2}\choose{2}}f_1^0+HF(N,d).$$

In \cite{ss2}, Schenck and Stillman place the graded $R-$module $N$ in the following short exact sequence of graded $R-$modules: $$0\rightarrow H_{\langle x,y,z\rangle}^0(N)\rightarrow N\rightarrow \oplus_{v\in\Delta_0^0}R/J(v)\rightarrow 0,$$ where $H_{\langle x,y,z\rangle}^0(N)$ is the zeroth local cohomology module of $N$ at the maximal ideal $\langle x,y,z\rangle$, $\Delta_0^0$ denotes the set of interior vertices of $\Delta$, and $$J(v)=\langle \ell_{v,1}^{r+1},\ldots,\ell_{v,n(v)}^{r+1}\rangle,$$ with $\ell_{v,i}$ being the linear forms defining the interior edges of different slopes with one of the vertices being $v$.

Theorem 3.1 in \cite{ss1} describes the minimal graded free resolution of $R/J(v)$, so $HF(R/J(v)),d)$ is known. In Corollary 4.5 in \cite{ss1} all of this information is combined to obtain $$\dim_{\mathbb R}C^r(\hat{\Delta})_d=L(\Delta,r,d)+HF(H_{\langle x,y,z\rangle}^0(N),d).$$ Since $H_{\langle x,y,z\rangle}^0(N)$ is a module of finite length, we have for $d$ sufficiently large $HF(H_{\langle x,y,z\rangle}^0(N),d)=0$. In fact, Alfeld and Schumaker show that it is enough to take $d\geq 3r+1$. The Schenck-Stiller conjecture claims that one can take $d\geq 2r+1$.

\subsection{The Schenck-Stiller example}

In this section we reduce our problem to calculating the Hilbert function of some ideal in degree $r$. We will provide a detailed picture of our approach which is also accessible to a non-specialist in this area.

Let $R=\mathbb R[x,y,z]$ be the ring of polynomials with real coefficients and let $\Delta_S:=\Delta$ be the following simplicial complex that triangulates a bounded connected region in the real plane:
\begin{center}
\epsfig{figure=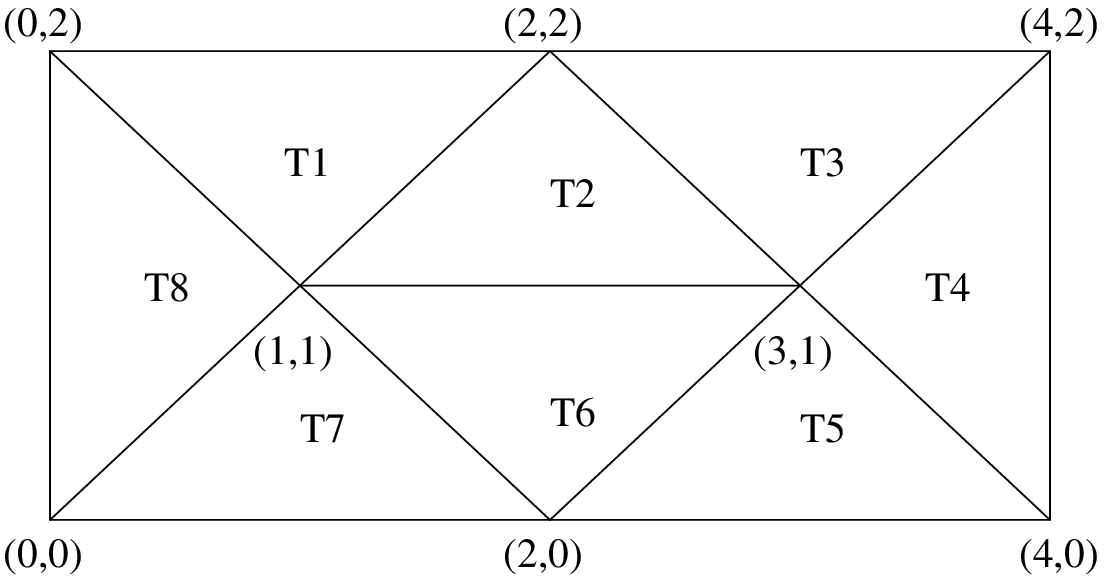,height=1.3in,width=2in}
\end{center}

The goal is to compute $\dim_{\mathbb R}C^r(\hat{\Delta})_d, d\geq 2r+1$ by showing that $HF(H_{\langle x,y,z\rangle}^0(N),d)=0, d\geq 2r+1$.

For $\Delta_S$, because it has exactly one totally interior edge, $H_{\langle x,y,z\rangle}^0(N)$ is isomorphic to $R/I$, where $I$ is a certain ideal (see the assertions before Note \ref{note:note}), and therefore it is enough to prove that $HF(H_{\langle x,y,z\rangle}^0(N),2r+1)=0$. Since our triangulation is specific, our analysis will be more direct, building on basic definitions which, via careful considerations, will lead us to our results.

Let $d=2r+1$, and let us label the triangles of $\Delta$ with $T_1,\ldots,T_8$, clockwise: for example the first triangle has vertices $(0,2),(2,2)$ and $(1,1)$. Let $\ell_{ij}$ be the linear form defining the common edge between the triangles $T_i$ and $T_j$. We have 9 interior edges, one of them being a non-pseudoedge (i.e., the edge common to the triangles $T_2$ and $T_6$; this edge is placed on the line in $\mathbb R^2$ of equation $y=1$, hence, after homogenization, $\ell_{26}=y-z$):
\begin{eqnarray} \ell_{12}&=&\ell_{78}=x-y\nonumber\\
\ell_{23}&=&\ell_{45}=x+y-4z\nonumber\\
\ell_{34}&=&\ell_{56}=x-y-2z\nonumber\\
\ell_{67}&=&\ell_{18}=x+y-2z\nonumber\\
\ell_{26}&=&y-z.\nonumber
\end{eqnarray}

An element $F\in C^r(\hat{\Delta})_d$ is an $8-$tuple: $$F=(F_1,\ldots,F_8),$$ where $F_i\in R_d$ are subject to the conditions:

\begin{eqnarray}
F_1-F_2&=&\ell_{12}^{r+1}G_{12}\nonumber\\
F_2-F_3&=&\ell_{23}^{r+1}G_{23}\nonumber\\
F_3-F_4&=&\ell_{34}^{r+1}G_{34}\nonumber\\
F_4-F_5&=&\ell_{45}^{r+1}G_{45}\nonumber\\
F_5-F_6&=&\ell_{56}^{r+1}G_{56}\nonumber\\
F_6-F_7&=&\ell_{67}^{r+1}G_{67}\nonumber\\
F_7-F_8&=&\ell_{78}^{r+1}G_{78}\nonumber\\
F_8-F_1&=&\ell_{18}^{r+1}G_{18}\nonumber\\
F_2-F_6&=&\ell_{26}^{r+1}G_{26},\nonumber
\end{eqnarray}where $G_{ij}$ are homogeneous polynomials in $R$ of degree $d-(r+1)=r$.

One should observe that once we know $F_1$ and the $G_{ij}$'s, $F$ is completely determined. The $G_{ij}$'s are subject to the following relations:
\begin{eqnarray}
(G_{12}+G_{78})\ell_{12}^{r+1}+(G_{67}+G_{18})\ell_{67}^{r+1}+G_{26}\ell_{26}^{r+1}=0,\nonumber\\
(-G_{23}-G_{45})\ell_{23}^{r+1}+(-G_{34}-G_{56})\ell_{34}^{r+1}+G_{26}\ell_{26}^{r+1}=0.\nonumber
\end{eqnarray}

From the discussions above, considering first the possibilities for $F_1$, then the possibilities for $G_{ij}$'s with $G_{26}$ fixed, and finally the possibilities for $G_{26}$ we see that the dimension we are looking for is $$\dim C^r(\hat{\Delta})_d={{d+2}\choose{2}}+4{{r+2}\choose{2}}+\epsilon,$$ where $\epsilon=\dim I_r$ with $$I=\langle \ell_{12}^{r+1},\ell_{67}^{r+1}\rangle :\ell_{26}^{r+1}\cap \langle \ell_{23}^{r+1},\ell_{34}^{r+1}\rangle :\ell_{26}^{r+1}.$$

The Alfeld-Schumaker formula says that $$L(\Delta,r,d)={{d+2}\choose{2}}+9{{d-r+1}\choose{2}}-2\left({{d+2}\choose{2}}-{{r+2}\choose{2}}\right)+\sigma,$$ and from the proof of Corollary 4.5 in \cite{ss1}, $$\sigma=2r\alpha-2\alpha^2,\mbox{ where }\alpha=\lfloor\frac{r+1}{2}\rfloor,$$ the smallest integer larger than $\frac{r+1}{2}$.

We have $d=2r+1$.

If $r+1=2n$, then $\sigma=2n^2-2n$. For this case $$\dim C^r(\hat{\Delta})_d-L(\Delta,r,d)=\epsilon-n.$$

If $r+1=2n+1$, then $\sigma=2n^2$. For this case  $$\dim C^r(\hat{\Delta})_d-L(\Delta,r,d)=\epsilon-(n+1).$$

So the Schenck-Stiller conjecture is true for this $\Delta$, if one can show that $\epsilon=n$ if $r+1=2n$, and $\epsilon=n+1$ if $r+1=2n+1$. Our goal is to prove these equalities.

\vskip .3in

In the language of the previous subsection, one needs to show that $$HF(H_{\langle x,y,z\rangle}^0(N),2r+1)=0.$$ Indeed our goal matches this new goal as we can see below.

First, by making the change of variables suggested in \cite{to}: $y-z=\frac{1}{2}\bar{y}$, $x-y=\bar{x}$, $x-y-2z=\bar{z}$, we can assume that $\ell_{26}=y$, $\ell_{12}=x$, $\ell_{67}=x+y$ and $\ell_{34}=z$, $\ell_{23}=x+z$. So the ideals $J(v_i)$ at the two interior vertices become: $$J(v_1)=\langle x^{r+1},(x+y)^{r+1},y^{r+1}\rangle,$$ and $$J(v_2)=\langle z^{r+1},(z+y)^{r+1},y^{r+1}\rangle.$$

Second, by \cite{ss1}, Theorem 3.1, the graded minimal free resolution for $R/J(v_i), i=1,2$ is $$0\rightarrow R(-a)\oplus R(-b)\stackrel{\psi_i}\rightarrow R^3(-r-1)\rightarrow R\rightarrow R/J(v_i)\rightarrow 0,$$ where $a$ and $b$ are some shifts (see Note \ref{note:note} below) and $\psi_i=\left(\begin{array}{cc} A_i&D_i\\ B_i&E_i\\C_i&F_i\end{array}\right).$ For example, if $i=1$ this means that $A_1x^{r+1}+B_1(x+y)^{r+1}+C_1y^{r+1}=0 \mbox{ and }D_1x^{r+1}+E_1(x+y)^{r+1}+F_1y^{r+1}=0.$ We should note that $$\langle C_1,F_1\rangle=\langle x^{r+1},(x+y)^{r+1}\rangle : y^{r+1},$$ and $$\langle C_2,F_2\rangle=\langle z^{r+1},(z+y)^{r+1}\rangle : y^{r+1}.$$

Third, by \cite{ss2}, $$H_{\langle x,y,z\rangle}^0(N)\approx R(-r-1)/\langle C_1,F_1,C_2,F_2\rangle.$$ With this, the Schenck-Stiller conjecture for this $\Delta$ reduces to showing that $$HF(R/\langle C_1,F_1,C_2,F_2\rangle,r)=0.$$

\vskip .2in

\begin{note}\label{note:note} By \cite{ss1}, Theorem 3.1, if $r+1=2n$, then $\deg(C_i)=\deg(F_i)=n,i=1,2$, and if $r+1=2n+1$, then $\deg(C_i)=n,\deg(F_i)=n+1, i=1,2$. Furthermore, the ideals $\langle C_i,F_i\rangle$ are complete intersections (see \cite{to}).
\end{note}

\vskip .2in

We have an exact sequence of $R-$ modules: $$0\rightarrow \frac{R}{\langle C_1,F_1\rangle\cap\langle C_2,F_2\rangle} \rightarrow \frac{R}{\langle C_1,F_1\rangle}\oplus \frac{R}{\langle C_2,F_2\rangle}\rightarrow \frac{R}{\langle C_1,F_1,C_2,F_2\rangle}\rightarrow 0.$$ Since $\langle C_i,F_i\rangle, i=1,2$ is a complete intersection, we have the graded minimal free resolution $$0\rightarrow R(-(\deg(C_i)+\deg(F_i))\rightarrow R(-\deg(C_i))\oplus R(-\deg(F_i))\rightarrow R\rightarrow R/\langle C_i,F_i\rangle\rightarrow 0,$$ and the Hilbert function of $\frac{R}{\langle C_1,F_1\rangle}\oplus \frac{R}{\langle C_2,F_2\rangle}$ can be computed using these resolutions.

Therefore, to prove the claim it will be enough to show \begin{eqnarray} HF(\frac{R}{\langle C_1,F_1\rangle\cap\langle C_2,F_2\rangle},r)&=&
2HF(\frac{R}{\langle C_1,F_1\rangle},r)\nonumber\\
&=& \left\{
                                          \begin{array}{ll}
                                           2n^2, & \hbox{if } r=2n-1; \\
                                           2n(n+1), & \hbox{if } r=2n.
                                          \end{array}
                                          \right.\nonumber
\end{eqnarray}

Equivalently, our main goal is to prove

\begin{thm}\label{thm:main} $$HF(\langle C_1,F_1\rangle\cap\langle C_2,F_2\rangle,r)=\epsilon= \left\{
                                          \begin{array}{ll}
                                           n, & \hbox{if } r=2n-1; \\
                                           n+1, & \hbox{if } r=2n.
                                          \end{array}
                                          \right.$$
\end{thm}

\section{Proof of the main result}

Let us denote $$K(r)=((\langle x^{r+1},(x+y)^{r+1}\rangle\cap\langle z^{r+1},(z+y)^{r+1}\rangle):y^{r+1})_r.$$ We want to show that $$\dim K(2n-1)=n\mbox{ and } \dim K(2n)=n+1.$$

We have the following sequence of useful lemmas.

\begin{lem}\label{lem:important0} For any $n\geq 1$ we have $$\dim K(2n-1)\geq n\mbox{ and } \dim K(2n)\geq n+1.$$
\end{lem}
\begin{proof} Since $HF(\frac{R}{\langle C_1,F_1,C_2,F_2\rangle},r)\geq 0$, then from the exact sequence above we have $HF(R,r)-\dim K(r)\leq 2HF(\frac{R}{\langle C_1,F_1\rangle},r)$ and hence the result.
\end{proof}

We have $$\langle x^{r+1},(x+y)^{r+1}\rangle :y^{r+1}=\langle C_1,F_1\rangle$$ and $$\langle z^{r+1},(z+y)^{r+1}\rangle :y^{r+1}=\langle C_2,F_2\rangle.$$

Observe that $C_1,F_1$ are polynomials in variables $x$ and $y$ so they are elements in $A=\mathbb R[x,y]$. $A$ is a subring of $R=\mathbb R[x,y,z]$. We will denote by $\langle C_1,F_1\rangle A$ the ideal in $A$ generated by $C_1,F_1$, and we will denote by $\langle C_1,F_1\rangle R$ the ideal in $R$ generated also by $C_1,F_1$. We also have $\langle C_1,F_1\rangle R\cap A=\langle C_1,F_1\rangle A$.

Similarly, $C_2,F_2$ are polynomials in variables $y$ and $z$, and so they belong to $B=\mathbb R[y,z]\subset R$. Again, $\langle C_2,F_2\rangle B$ will denote the ideal in $B$ generated by $C_2,F_2$, and  $\langle C_2,F_2\rangle R$ will denote the ideal in $R$ generated by $C_2,F_2$.

\begin{lem}\label{lem:important1} Let $F\in K(r)$. Then
\begin{enumerate}
\item If we write $F=z^rf_0+z^{r-1}f_1+\cdots+ zf_{r-1}+f_r$ with $f_i\in A=\mathbb R[x,y]$, then $f_i\in \langle C_1,F_1\rangle A$.
\item If we write $F=x^rg_0+x^{r-1}g_1+\cdots+ xg_{r-1}+g_r$ with $g_j\in B=\mathbb R[y,z]$, then $g_j\in \langle C_2,F_2\rangle B$.
\end{enumerate}
\end{lem}
\begin{proof} It is enough to prove the first part, since we can switch the roles of $x$ and $z$ afterwards. Let $$F=z^rf_0+z^{r-1}f_1+\cdots+ zf_{r-1}+f_r\in K(r).$$

Since $F\in \langle C_1,F_1\rangle R$, substituting $z=0$ we obtain $f_r\in\langle C_1,F_1\rangle A$. Therefore $$z(z^{r-1}f_0+\cdots+f_{r-1})\in \langle C_1,F_1\rangle R.$$ Since $\langle C_1,F_1\rangle A\subset A$ is a complete intersection, then $\{z,C_1,F_1\}$ is an $R-$regular sequence, and therefore  $z^{r-1}f_0+\cdots+f_{r-1}\in \langle C_1,F_1\rangle R$. Again, by making $z=0$, and with the same argument as above, we obtain inductively that for all $i=0,\ldots,r$, one has $f_i\in\langle C_1,F_1\rangle A$. \end{proof}

A useful consequence of the proof of the above lemma is that if $F\in K(r)$ and if $x^ay^bz^c$ is a nonzero monomial in the expression of $F$, since the ideals $\langle C_i,F_i\rangle, i=1,2$ are generated in degree $\geq n$, then $a\leq r-n$ and $c\leq r-n$.

The key result is the following lemma. For convenience we assume that $0$ has any degree. In this section most of the arguments use induction on $r$ (the base cases are simple calculations, and we left them out of the notes), and the third part of this next result will help resolve the inductive step.

\begin{lem}\label{lem:important} The following statements are true:
\begin{enumerate}
  \item $y^r\in K(r)$.
  \item For any $j\leq r-1$, $y^j\notin \langle x^{r+1},(x+y)^{r+1}\rangle:y^{r+1}$ and $y^j\notin \langle z^{r+1},(z+y)^{r+1}\rangle:y^{r+1}$.
  \item If $F\in K(r)$, then $y\frac{\partial^2F}{\partial x\partial z}=yF_{xz}\in K(r-1)$.
\end{enumerate}
\end{lem}
\begin{proof} Let $I=\langle x^{r+1},(x+y)^{r+1}\rangle:y^{r+1}=\langle C_1,F_1\rangle\subset A=\mathbb R[x,y]$.
\vskip .2in
(1). From the exact sequence $$0\rightarrow A(-(\deg(C_1)+\deg(F_1))\rightarrow A(-\deg(C_1))\oplus A(-\deg(F_1))\rightarrow A\rightarrow A/\langle C_1,F_1\rangle\rightarrow 0,$$ we have $$HF(A/I,r)=r+1-(r+1-\deg(C_1) + r+1-\deg(F_1))+(r+1-\deg(C_1)-\deg(F_1))=0,$$ so $y^r\in I$. Similarly, replacing $x$ with $z$, we have $y^r\in K(r)$.
\vskip .2in
(2). Let $j\leq r-1$. If $y^j\in \langle x^{r+1},(x+y)^{r+1}\rangle:y^{r+1}$, then there exist $P,Q\in A$ of degree $j$ such that $$y^{j+r+1}=Px^{r+1}+Q(x+y)^{r+1}.$$ Taking the partial derivative with respect to $x$ we obtain $$0=((r+1)P+xP_x)x^r+((r+1)Q+(x+y)Q_x)(x+y)^r.$$ Therefore $(r+1)P+xP_x$, which has degree $j$, is a multiple of $(x+y)^r$. But $j\leq r-1$, and therefore $$(r+1)P+xP_x=0.$$ If $P=\sum a_ix^iy^{j-i}$, we have $$a_0=0\mbox{ and }(r+1+i)a_i=0, i\geq 1.$$ Hence $P=0$, and therefore we have a contradiction.
\vskip .2in
(3). Let $F\in K(r)$. So there exist $P,Q\in R$ such that $$y^{r+1}F=Px^{r+1}+Q(x+y)^{r+1}.$$ Taking the partial derivative with respect to $z$ we have $$y^{r+1}F_z=P_zx^{r+1}+Q_z(x+y)^{r+1}.$$ Taking the partial derivative with respect to $x$ we have $$y^r(yF_{xz})\in\langle x^r,(x+y)^r\rangle.$$ Similarly, we have $$y^r(yF_{xz})\in\langle z^r,(z+y)^r\rangle.$$ Since $\deg(yF_{xz})=1+r-2=r-1$, we indeed obtain that $yF_{xz}\in K(r-1)$.
\end{proof}

\vskip .2in

\begin{prop}\label{prop:symmetry} We have $$K(r)=W:=\{f\in (\langle x^{r+1},(x+y)^{r+1}\rangle:y^{r+1})_r|f(x,y,z)=f(z,y,x)\}.$$
\end{prop}
\begin{proof} For convenience denote with $\bar{f}(x,y,z)=f(z,y,x)$.

The inclusion ``$\supseteq$'' is immediate, by switching $x$ and $z$, so $\dim K(r)\geq \dim W$.

For the other inclusion, consider $\{f_1,\ldots,f_m\}$ to be a basis for $K(r)$. Since $f_i\in K(r)$, then $\bar{f}_i\in K(r)$, and therefore $g_i:=f_i+\bar{f}_i\in K(r)$. Since $g_i=\bar{g}_i$, and since $$g_i\in K(r)\subset (\langle x^{r+1},(x+y)^{r+1}\rangle:y^{r+1})_r,$$ we have that $g_i\in W, i=1,\ldots,m$.

Suppose that these $g_i$ are linearly dependent. Then there exist $c_1,\ldots,c_m\in\mathbb R$ not all zero, such that $$c_1(f_1+\bar{f}_1)+\cdots+c_m(f_m+\bar{f}_m)=0.$$

We have $F:=c_1f_1+\cdots+c_mf_m\in K(r)$ and $F=-\bar{F}$. We show by induction on $r$ that in these conditions $F=0$. From Lemma \ref{lem:important} (3), we have $yF_{xz}\in K(r-1)$. Also $yF_{xz}=-y(\bar{F})_{xz}$, and by induction $$yF_{xz}=0,$$ which means $$F=P(x,y)+Q(y,z)=a_n(x^{r-n}-z^{r-n})y^n+a_{n+1}(x^{r-n-1}-z^{r-n-1})y^{n+1}+\cdots+a_{r-1}(x-z)y^{r-1}.$$

From Lemma \ref{lem:important1} we have that for $j=n,\ldots,r-1$, $$a_jy^j\in \langle z^{r+1},(z+y)^{r+1}\rangle:y^{r+1},$$ and therefore from Lemma \ref{lem:important} (2), $a_j=0$.

This means $F=0$, and since $f_1,\ldots,f_m$ are linearly independent, all the $c_i$ must vanish. But this implies that $\dim W\geq m=\dim K(r)$, and with the previous inclusion we obtain the desired equality. \end{proof}

\begin{prop}\label{prop:mindegree} The ideal $$\mathcal I(r):=(\langle x^{r+1},(x+y)^{r+1}\rangle:y^{r+1})\cap (\langle z^{r+1},(z+y)^{r+1}\rangle:y^{r+1})$$ is minimally generated in degree $r$.
\end{prop}
\begin{proof} We show by induction on $r\geq 2$, that $\dim_{\mathbb R} \mathcal I(r)_{r-1}=0$.

Let $F\in \mathcal I(r)_{r-1}$. The same proof as for Lemma \ref{lem:important} (3), gives us that $$yF_{xz}\in \mathcal I(r-1)_{r-2}.$$ By induction this must vanish and therefore $$F_{xz}=0.$$

So $$F=a_{r-1}x^{r-1}+a_{r-2}x^{r-2}y+\cdots+a_1xy^{r-2}+a_0y^{r-1}+b_1zy^{r-2}+\cdots+b_{r-1}z^{r-1}\in \mathcal I(r).$$ A similar proof as for Lemma \ref{lem:important1} will yield $$a_{r-1},a_{r-2}y,\ldots,a_1y^{r-2}\in \langle z^{r+1},(z+y)^{r+1}\rangle:y^{r+1}.$$ These must vanish because of Lemma \ref{lem:important} (2). Similarly, $b_1=\cdots=b_{r-1}=0$. So $F=a_0y^{r-1}\in \langle x^{r+1},(x+y)^{r+1}\rangle:y^{r+1}$, and this must vanish as well, from the same lemma. \end{proof}

Now we can prove our desired equalities.

\begin{thm}\label{thm:odd} For all $n\geq 2$, $$\dim K(2n-1)=n.$$
\end{thm}
\begin{proof} From Lemma \ref{lem:important0}, it will be enough to show that $\dim K(2n-1)\leq n$.

As we denoted before, let $$\mathcal I(2n-1)=(\langle x^{2n},(x+y)^{2n}\rangle:y^{2n})\cap (\langle z^{2n},(z+y)^{2n}\rangle:y^{2n}).$$ Then we need to show that $$HF(\mathcal I(2n-1),2n-1)\leq n.$$

Let $\succ$ be the Graded Reverse Lexicographic order on the monomials of $R=\mathbb R[x,y,z]$, with $x\succ y\succ z$. Let $$in_{\succ}(\mathcal I(2n-1))$$ be the initial ideal of $\mathcal I(2n-1)$ with respect to $\succ$.

From Proposition \ref{prop:mindegree}, $\mathcal I(2n-1)$ is minimally generated in degree $2n-1$. Also, since by Proposition \ref{prop:symmetry}, the polynomials of degree $2n-1$ in $\mathcal I(2n-1)$ are symmetric in $x$ and $z$, we have that $$\{y^{2n-1},xy^{2n-2},\ldots,x^{n-1}y^n, xzM_1,\ldots,xzM_p\}$$ includes the monomials that generate $in_{\succ}(\mathcal I(2n-1))$. Here we also used the remark after Lemma \ref{lem:important1} that the power of $x$ in a nonzero monomial of an element of degree $2n-1$ in $\mathcal I(2n-1)$ is $\leq r-n=2n-1-n=n-1$.

Next we show that if $xzM$ is the leading monomial of an element $F\in \mathcal I(2n-1)$, then $\deg(F)\geq 2n$. Let $F\in \mathcal I(2n-1)_{2n-1}=K(2n-1)$ be such that $in_{\succ}(F)=xzM$. By Proposition \ref{prop:symmetry}, $F$ is symmetric in $x$ and $z$, so $$F=xzG,$$ for some $G\in R$.


We have $$y^{2n}(xzG)=Ax^{2n}+B(x+y)^{2n}\mbox{ and } y^{2n}(xzG)=Cz^{2n}+D(z+y)^{2n}.$$ We have $\{z,x^{2n},(x+y)^{2n}\}$ and $\{x,z^{2n},(z+y)^{2n}\}$ are $R-$regular sequences so $$y^{2n}(xG)=A'x^{2n}+B'(x+y)^{2n}\mbox{ and } y^{2n}(zG)=C'z^{2n}+D'(z+y)^{2n}.$$ From this we have that $x|B'$ and $z|D'$, and therefore $$y^{2n}G=A'x^{2n-1}+B''(x+y)^{2n}\mbox{ and } y^{2n}G=C'z^{2n-1}+D''(z+y)^{2n}.$$

Taking the partial derivative with respect to $y$ of the two equations above, we obtain

$$y^{2n-1}(2nG+yG_y)\in\langle x^{2n-1},(x+y)^{2n-1}\rangle\mbox{ and } y^{2n-1}(2nG+yG_y)\in\langle z^{2n-1},(z+y)^{2n-1}\rangle.$$

We obtained that $2nG+yG_y$ is an element of degree $2n-3$ in $\mathcal I(2n-2)$, or is equal to zero. So from Proposition \ref{prop:mindegree}, $$2nG+yG_y=0.$$ If $G=\sum a_{i,j,k}x^iy^jz^k$, then $(2n+j)a_{i,j,k}=0$, and so $G=0$.

We obtained that the leading monomials of elements of degree $2n-1$ in $\mathcal I(2n-1)$, belong to the set $\{y^{2n-1},xy^{2n-2},\ldots,x^{n-1}y^n\}$, and therefore $$HF(\mathcal I(2n-1),2n-1)\leq n.$$ \end{proof}

\begin{thm}\label{thm:even} For all $n\geq 2$, $$\dim K(2n)=n+1.$$
\end{thm}
\begin{proof} From Lemma \ref{lem:important0}, it will be enough to show that $\dim K(2n)\leq n+1$.

Since $y^{2n}\in K(2n)$, from Lemma \ref{lem:important} (1), we can find a basis for $K(2n)$: $$\{y^{2n},H_1,\ldots,H_m\}.$$

Suppose $m\geq n+1$. From Lemma \ref{lem:important} (3), $y(H_1)_{xz},\ldots,y(H_m)_{xz}\in K(2n-1)$. From Theorem \ref{thm:odd}, we have $\dim K(2n-1)=n$ and therefore, these elements must be linearly dependent. So there exist constants $c_1,\ldots,c_m\in\mathbb R$, not all zero, such that $$c_1y(H_1)_{xz}+\cdots+c_my(H_m)_{xz}=0.$$ This implies that $$(c_1H_1+\cdots+c_mH_m)_{xz}=0.$$

So $H:=c_1H_1+\cdots+c_mH_m\in K(2n)$, which is symmetric in $x$ and $z$ from Proposition \ref{prop:symmetry}, is of the form $$H=a_{2n}(x^{2n}+z^{2n})+\cdots+a_2(x^2+z^2)y^{2n-2}+a_1(x+z)y^{2n-1}+a_0y^{2n}.$$

Using Lemma \ref{lem:important1} and Lemma \ref{lem:important} (2), we obtain $$a_1=\cdots=a_{2n}=0,$$ which leads to the linear dependency $$-a_0y^{2n}+c_1H_1+\cdots+c_mH_m=0.$$ This is a contradiction since $\{y^{2n},H_1,\ldots,H_m\}$ is a basis for $K(2n)$. We obtain that $m\leq n$ which proves the theorem: $$\dim K(2n)\leq n+1.$$ \end{proof}

\section{Connections with Schur functors and Roth's equation}

Our initial approach to prove Theorem \ref{thm:main} was more direct: consider an arbitrary element in $\langle C_1,F_1\rangle_r$ and we require it to belong to $\langle C_2,F_2\rangle_r$; this will lead to a comparison of polynomial coefficients.

\vskip .2in

Let $$F=\sum_{i+j+k=r}a_{i,j,k}x^iy^jz^k$$ be an element in $\langle C_1,F_1\rangle\cap\langle C_2,F_2\rangle\subset R=\mathbb R[x,y,z]$ of degree $r$.

We can write $$(*)\mbox{ }F=z^rf_0+z^{r-1}f_1+\cdots+ zf_{r-1}+f_r,$$ where for $k=0,\ldots,r,$ $$f_k=a_{0,k,r-k}y^k+a_{1,k-1,r-k}xy^{k-1}+\cdots+a_{k,0,r-k}x^k$$ is a homogeneous polynomial of degree $k$ in $A=\mathbb R[x,y]$.

From Lemma \ref{lem:important1}, we obtain that for all $k=0,\ldots,r$, $$f_k\in\langle C_1,F_1\rangle A.$$

Since $\langle C_1,F_1\rangle A$ is minimally generated in degree $\geq n$ (see Note \ref{note:note}), then we obtain $$f_0=\cdots=f_{n-1}=0.$$ We are going to use these equations later when we switch the roles of $x$ and $z$. For now let us consider $k=n,\ldots,r$. We have

$$f_k\in\langle C_1,F_1\rangle=\langle x^{r+1},(x+y)^{r+1}\rangle:y^{r+1},$$ and therefore $$(**)\mbox{ }y^{r+1}f_k=P_kx^{r+1}+Q_k(x+y)^{r+1},$$ where $P_k=\sum_{i=0}^kp_{i,k-i}x^iy^{k-i}$ and $Q_k=\sum_{i=0}^kq_{i,k-i}x^iy^{k-i}$ are some homogeneous polynomials of degree $k$ in $A$.

Writing $(x+y)^{r+1}=\sum_{i=0}^{r+1}{{r+1}\choose{i}}x^iy^{r+1-i}$ and identifying the coefficients of $x^uy^v$ in both the left and right-hand sides of $(**)$, for each $k=n,\ldots,r$, we obtain:

\vskip .1in

\noindent (1) For the monomial $x^uy^v$ with $0\leq u\leq k$ we can determine the coefficients of $f_k$ from the coefficients of $Q_k$, as follows:

$$\left(\begin{array}{c}a_{0,k,r-k}\\a_{1,k-1,r-k}\\\vdots\\a_{k,0,r-k}\end{array}\right)= \left(
\begin{array}{ccccc}
1&0&\cdots&0&0\\
{{r+1}\choose{1}}&1&\cdots&0&0\\
{{r+1}\choose{2}}&{{r+1}\choose{1}}&\cdots&0&0\\
\vdots&\vdots& &\vdots&0\\
{{r+1}\choose{k}}&{{r+1}\choose{k-1}}&\cdots&{{r+1}\choose{1}}&1
\end{array}
\right)\left(\begin{array}{c}q_{0,k}\\q_{1,k-1}\\\vdots\\q_{k,0}\end{array}\right).$$

\vskip .1in

\noindent (2) For the monomial $x^uy^v$ with $k+1\leq u\leq r$ we obtain the following conditions on the coefficients of $Q_k$:

$$\left(
\begin{array}{cccc}
{{r+1}\choose{k+1}}&{{r+1}\choose{k}}&\cdots&{{r+1}\choose{1}}\\
{{r+1}\choose{k+2}}&{{r+1}\choose{k+1}}&\cdots&{{r+1}\choose{2}}\\
\vdots&\vdots& &\vdots\\
{{r+1}\choose{r}}&{{r+1}\choose{r-1}}&\cdots&{{r+1}\choose{r-k}}
\end{array}
\right)\left(\begin{array}{c}q_{0,k}\\q_{1,k-1}\\\vdots\\q_{k,0}\end{array}\right)= \left(\begin{array}{c}0\\0\\\vdots\\0\end{array}\right).$$

\vskip .1in

\noindent (3) For the monomial $x^uy^v$ with $u\geq r+1$ the coefficient identifications will express the coefficients of $P_k$ in terms of the coefficients of $Q_k$: $$\left(\begin{array}{c}p_{0,k}\\p_{1,k-1}\\\vdots\\p_{k,0}\end{array}\right)= -\left(
\begin{array}{ccccc}
1&{{r+1}\choose{r}}&{{r+1}\choose{r-1}}&\cdots&{{r+1}\choose{r+1-k}}\\
0&1&{{r+1}\choose{r}}&\cdots&{{r+1}\choose{r+1-(k-1)}}\\
\vdots&\vdots&\vdots& &\vdots\\
0&0&0&\cdots&1
\end{array}
\right)\left(\begin{array}{c}q_{0,k}\\q_{1,k-1}\\\vdots\\q_{k,0}\end{array}\right).$$ These equations will not be useful for our computations. Basically they show how to create $P_k$ from $Q_k$ to have $(**)$ be valid.

\vskip .2in

Let us consider our initial polynomial $F$ expanded by the powers of $x$: $$(***)\mbox{ } F=x^rg_0+x^{r-1}g_1+\cdots+xg_{r-1}+g_r,$$ where $$g_i=a_{r-i,i,0}y^i+a_{r-i,i-1,1}y^{i-1}z+\cdots+a_{r-i,0,i}z^i\in B=\mathbb R[z,y]$$ are homogeneous of degree $i$. Using the fact that $F\in\langle C_2,F_2\rangle R$, Lemma \ref{lem:important1} yields $g_i\in\langle C_2,F_2\rangle B$. Since this ideal is minimally generated in degree $n$ (see Note \ref{note:note}) we have $g_0=\cdots=g_{n-1}=0$. Equivalently,
\begin{eqnarray}
a_{r,0,0}&=&0\nonumber\\
a_{r-1,1,0}&=&a_{r-1,0,1}=0\nonumber\\
&\vdots&\nonumber\\
a_{r-(n-1),n-1,0}&=&a_{r-(n-1),n-2,1}=\cdots=a_{r-(n-1),0,n-1}=0.\nonumber
\end{eqnarray}

These mean that for each $k=n,\ldots,r$, the last $n+k-r$ of the coefficients $a_{i,j,k}$ in (1) above must vanish. So for each $k=n,\ldots,r$, we obtain $n+k-r$ more linear relations among the parameters $q_{a,b}$, that combined with the relations already obtained in (2), yield that for each $k=n,\ldots,r$, the vector $\left(\begin{array}{c}q_{0,k}\\q_{1,k-1}\\\vdots\\q_{k,0}\end{array}\right)$ is in the kernel of a $n\times (k+1)$ matrix $\mathcal M(k)$. We will see this matrix in more detail later on when we study the cases $r=2n-1$ and $r=2n$.

What remains from (1) are the first $(k+1)-(n+k-r)=r-n+1$ of the coefficients $a_{i,j,k}$. So for all $k=n,\ldots,r$,

$$\left(\begin{array}{c}a_{0,k,r-k}\\a_{1,k-1,r-k}\\\vdots\\a_{r-n,k-(r-n),r-k}\end{array}\right)= \left(
\begin{array}{ccccc}
1&0&\cdots&0&0\\
{{r+1}\choose{1}}&1&\cdots&0&0\\
{{r+1}\choose{2}}&{{r+1}\choose{1}}&\cdots&0&0\\
\vdots&\vdots& &\vdots&0\\
{{r+1}\choose{r-n}}&{{r+1}\choose{r-n-1}}&\cdots&{{r+1}\choose{1}}&1
\end{array}
\right)\left(\begin{array}{c}q_{0,k}\\q_{1,k-1}\\\vdots\\q_{r-n,k-(r-n)}\end{array}\right).$$

Combining the above for each $k=n,\ldots,r$ we obtain

$$\left(
\begin{array}{ccc}
a_{0,n,r-n}&\cdots&a_{0,r,0}\\
a_{1,n-1,r-n}&\cdots&a_{1,r-1,0}\\
\vdots& &\vdots\\
a_{r-n,2n-r,r-n}&\cdots&a_{r-n,n,0}
\end{array}
\right)=\mathcal D\cdot \left(
\begin{array}{ccc}
q_{0,n}&\cdots&q_{0,r}\\
q_{1,n-1}&\cdots&q_{1,r-1}\\
\vdots& &\vdots\\
q_{r-n,2n-r}&\cdots&q_{r-n,n}
\end{array}
\right),$$ where $\mathcal D$ is the lower-triangular matrix\footnote{A square matrix is called {\em lower (upper) triangular} if the entries above (below) the main diagonal are all zero.} in the previous vector equation. Denote with $\mathcal A$ the matrix to the left of the above equation and with $\mathcal Q$ the matrix of $q_{a,b}$'s.

$$\mathcal A=\mathcal D\cdot\mathcal Q.$$

\vskip .2in

At this moment we interchange the roles of $x$ and $z$, and we start with $F$ expanded by the powers of $x$, as we've seen in $(***)$. At this point we mentioned that $g_i\in\langle C_2,F_2\rangle B$, and therefore for $i=n,\ldots,r,$ $$g_i\in\langle z^{r+1},(z+y)^{r+1}\rangle:y^{r+1}.$$ Similarly as before $$y^{r+1}g_i=R_ix^{r+1}+S_i(x+y)^{r+1},$$ where $R_i=\sum_{k=0}^ir_{i-k,k}y^{i-k}z^k$ and $S_i=\sum_{k=0}^is_{i-k,k}y^{i-k}z^k$ are some homogeneous polynomials of degree $i$ in $B$.

With the same arguments as before (here we use the fact that $f_0=\cdots=f_{n-1}=0$), we obtain that for $i=n,\ldots,r$, the vector $\left(\begin{array}{c}s_{i,0}\\s_{i-1,1}\\\vdots\\s_{0,i}\end{array}\right)$ is in the kernel of the same $n\times (i+1)$ matrix $\mathcal M(i)$. Furthermore, we have for all $i=n,\ldots,r$,

$$\left(\begin{array}{c}a_{r-i,i,0}\\a_{r-i,i-1,1}\\\vdots\\a_{r-i,i-(r-n),r-n}\end{array}\right)= \left(
\begin{array}{ccccc}
1&0&\cdots&0&0\\
{{r+1}\choose{1}}&1&\cdots&0&0\\
{{r+1}\choose{2}}&{{r+1}\choose{1}}&\cdots&0&0\\
\vdots&\vdots& &\vdots&0\\
{{r+1}\choose{r-n}}&{{r+1}\choose{r-n-1}}&\cdots&{{r+1}\choose{1}}&1
\end{array}
\right)\left(\begin{array}{c}s_{i,0}\\s_{i-1,1}\\\vdots\\s_{i-(r-n),r-n}\end{array}\right).$$

Combining these vector equations for all $i=n,\ldots,r$ we obtain

$$\left(
\begin{array}{ccc}
a_{r-n,n,0}&\cdots&a_{0,r,0}\\
a_{r-n,n-1,1}&\cdots&a_{0,r-1,1}\\
\vdots& &\vdots\\
a_{r-n,2n-r,r-n}&\cdots&a_{0,n,r-n}
\end{array}
\right)=\mathcal D\cdot \left(
\begin{array}{ccc}
s_{n,0}&\cdots&s_{r,0}\\
s_{n-1,1}&\cdots&s_{r-1,1}\\
\vdots& &\vdots\\
s_{2n-r,r-n}&\cdots&s_{n,r-n}
\end{array}
\right),$$ where $\mathcal D$ is the same invertible matrix as before. Denote with $\mathcal B$ the matrix to the left of this matrix equation and with $\mathcal S$ the matrix of $s_{u,v}$'s.

$$\mathcal B=\mathcal D\cdot\mathcal S.$$

Observe that the entries of $\mathcal B$ are the same as the entries of $\mathcal A$, but in different positions. For example, the last column of $\mathcal B$ is the first row of $\mathcal A$ written backwards. This pattern is true for all the columns of $\mathcal B$. In matrix form this relation can be expressed as $$\mathcal J\cdot\mathcal A^{T}\cdot\mathcal J=\mathcal B,$$ where $\mathcal J$ is the exchange matrix $\left(
\begin{array}{cccc}
0&\cdots&0&1\\
0&\cdots&1&0\\
\vdots& & &\vdots\\
1&\cdots&0&0
\end{array}
\right)$.

Everything put together gives $$\mathcal D\cdot\mathcal S=\mathcal J\cdot\mathcal Q^{T}\cdot\mathcal D^{T}\cdot\mathcal J.$$

\vskip .1in

To summarize, we parameterized the elements in $\langle C_1,F_1\rangle_r\cap\langle C_2,F_2\rangle_r$ by two sets of parameters $q_{a,b}$ and $s_{u,v}$, both in the kernel $K$ of the same matrix, and with $(r-n+1)^2$ relations among them given by the above matrix equation. As we will see in the next subsections, the real challenge is not to find $\dim K$ (we will use powerful results from representation theory to do this), but it is to answer the following question: given any two matrices $\mathcal S$ and $\mathcal Q$ satisfying the above matrix equation, can these two matrices be extended to two sets of parameters $q_{a,b}$ and $s_{u,v}$ that are in the kernel $K$?

\vskip .3in

To find the dimension of the kernel $K$ of the matrix with diagonal blocks $\mathcal M(n),\ldots,\mathcal M(r)$ one has to use relevant facts about Schur functors. We briefly recall some of these in the next subsection. We follow the nice exposition in \cite{e}, A.2.5, but we also refer the reader to \cite{m}, Chapter 1.

\subsection{Schur functors} Let $V$ be a vector space over a field of characteristic 0. Let $\dim V=t$. The finite dimensional representations of $SL(V)$ decompose into a direct sum of irreducible representations. These summands are called {\em Schur modules} and can be viewed as functors (called {\em Schur functors}). To a sequence of numbers $t>d_1\geq\cdots\geq d_s>0$ one can associate the Schur module $S^{\{d_1,\ldots,d_s\}}V$, which is a nontrivial, irreducible finite-dimensional representation of $SL(V)$.

Let $Y$ be a Young diagram containing $s$ rows of boxes, row $i$ having $d_i$ boxes. Label the boxes in matrix notation; thus row $i$ has the boxes labeled $i1,i2,\ldots,id_i$. Set $A_i=\{i1,i2,\ldots,id_i\}$ and $A=A_1\cup A_2\cup\cdots\cup A_s$.

Similarly, we label the columns of $Y$ as $B_1,\ldots,B_{d_1}$. For example $B_1=\{11,21,\ldots,s1\}$. Thus $A=B_1\cup\cdots\cup B_{d_1}$.

For each index $ij\in A$ we consider a copy $V_{ij}\cong V$, viewed as a representation of $G=SL(V)$. For each set $A_i$ (or $B_j$) as above, we can consider $S_{A_i}V,\Lambda^{A_i}V,T_{A_i}V$ symmetric, exterior or tensor products of copies of $V$ labeled by indices in $A_i$ (or $B_j$).

For the sake of simplicity, assume $C=\{1,2,\ldots,m\}$ and let us recall that we have an embedding $$\Lambda^cV\rightarrow T_cV,$$ given via $$v_1\wedge\cdots\wedge v_c\mapsto \sum_{\sigma\in S_c}sgn(\sigma)v_{\sigma(1)}\otimes\cdots\otimes v_{\sigma(c)},$$ where $S_c$ is the symmetric group on $c$ letters.

Thus considering the induced natural maps, let $\Phi$ be the composition $$\Lambda^{A_1}V\otimes\cdots\otimes\Lambda^{A_s}V\rightarrow T_{A_1}V\otimes\cdots\otimes T_{A_s}V\cong T_{B_1}V\otimes\cdots\otimes T_{B_{d_1}}V\rightarrow S_{B_1}V\otimes\cdots\otimes S_{B_{d_1}}V.$$

We set $S^{\{d_1,\ldots,d_s\}}$ to be the image of $\Phi$. The classical convention to create $S^{\{d_1,\ldots,d_s\}}V$ from the Young diagram $Y$ is to take antisymmetric products on columns and symmetric product on rows. Observe that in \cite{e}, the construction is reversed.

The action of $G=SL(V)$ on $V$ extends naturally to the action of $G$ on $T_cV$ for each $c\in\mathbb N$, by acting $\sigma\in G$ on each factor, and this action induces an action of $G$ on $S^{\{d_1,\ldots,d_s\}}$.

\begin{exm} $S^{\{2,1\}}V$. The Young tableau is

\begin{center}\epsfig{figure=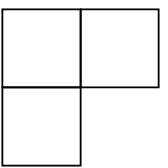,height=0.5in,width=0.5in}\end{center}

We have $A_1=\{11,12\},A_2=\{21\},B_1=\{11,21\},B_2=\{12\}$. So $A=A_1\cup A_2=B_1\cup B_2$.

$\Phi$ is the composition $$\Lambda^{A_1}V\otimes \Lambda^{A_2}V\stackrel{\alpha}\rightarrow T_{A_1}V\otimes T_{A_2}V \cong T_{B_1}V\otimes T_{B_2}V\stackrel{\beta}\rightarrow S_{B_1}V\otimes S_{B_2}V,$$ where $$\alpha((v_{11}\wedge v_{12})\otimes v_{21})=v_{11}\otimes v_{12}\otimes v_{21}-v_{12}\otimes v_{11}\otimes v_{21},$$ and therefore

$$\Phi(v_{11}\wedge v_{12}\otimes v_{21})= v_{11}\cdot v_{21}\otimes v_{12}-v_{12}\cdot v_{21}\otimes v_{11}.$$ $v_{11}\cdot v_{21}$ and $v_{12}\cdot v_{21}$ are products in $S_2V$. We have the fact that $Im(\Phi)=S^{\{2,1\}}V$.

Skipping details that can be found in \cite{fh}, Chapter 1, Lecture Six, one can show that $$\Lambda^2V\otimes V\simeq \Lambda^3V\oplus S^{\{2,1\}}V.$$ One idea of showing this is to embed everything in $T_3V$, and use the conventions and the approach in \cite{fh}. This is based on showing that $S^{\{2,1\}}V$ is the kernel of the map: $$\rho:\Lambda^2V\otimes V\rightarrow\Lambda^3V,$$ given by $\rho(v_1\wedge v_3\otimes v_2)=v_1\wedge v_2\wedge v_3$.

Hence we obtain $$\dim S^{\{2,1\}}V=\frac{(t-1)t(t+1)}{2}.$$
\end{exm}

In fact, in general the formula for $\dim S^{\{d_1,\ldots,d_s\}}V$ is the determinant of the matrix $M$, where $M$ is the $s\times s$ matrix with $$M_{i,j}={{t}\choose{d_j+i-j}}.$$ (Observe that in \cite{e}, Theorem A 2.9, there is a misprint in the entry $M_{s,s-1}$ where it should read ${{t}\choose{d_{s-1}+1}}$.)

\vskip .2in

\subsection{Finding $\dim K$} We have to study the two cases: $r=2n-1$ and $r=2n$.

\subsubsection{The case $r=2n-1$.} If $r=2n-1$, for $k=n,\ldots, 2n-1$ we have
$$\mathcal M(k)=\left(\begin{array}{cccccccc}
m_n&m_{n-1}&\cdots&m_1&m_0&0&\cdots&0\\
m_{n+1}&m_{n}&\cdots&m_2&m_1&m_0&\cdots&0\\
\vdots&\vdots& &\vdots&\vdots&\vdots& &\vdots\\
m_{2n-1}&m_{2n-2}&\cdots&m_n&m_{n-1}&m_{n-2}&\cdots&m_{2n-1-k}
\end{array}\right)$$ is an $n\times(k+1)$ matrix. We denoted $m_{\ell}={{2n}\choose{\ell}},\ell=0,\ldots,2n-1$.

We also have $$K=\oplus_{k=n}^{2n-1}Ker(\mathcal M(k)).$$

\begin{lem} We have $$\dim K=\frac{n(n+1)}{2}.$$
\end{lem}
\begin{proof} First we have $$\dim K=\sum_{k=n}^{2n-1}\dim Ker(\mathcal M(k))=\sum_{k=n}^{2n-1}(k+1-rank(\mathcal M(k))).$$

One should notice the $n\times n$ block in the left part of $\mathcal M(k)$:
$$\mathcal N=\left(\begin{array}{cccc}
{{2n}\choose{n}}&{{2n}\choose{n-1}}&\cdots&{{2n}\choose{1}}\\
{{2n}\choose{n+1}}&{{2n}\choose{n}}&\cdots&{{2n}\choose{2}}\\
\vdots&\vdots&\ddots&\vdots\\
{{2n}\choose{2n-1}}&{{2n}\choose{2n-2}}&\cdots&{{2n}\choose{n}}
\end{array}\right).$$

We have $m_{\ell}=m_{2n-\ell}$. From the above considerations, the determinant of this block is the dimension of the nontrivial, irreducible representation $S^{\{\lambda_1,\ldots,\lambda_n\}}V, \lambda_i=n,$ of $SL(V)$, where $V$ is a vector space of dimension $2n$.

With this fact, since these representations are nontrivial, these determinants are never zero, and therefore, for all $k=n,\ldots,2n-1$, $rank(\mathcal M(k))=n$. So $$\dim K=\sum_{k=1}^n k=\frac{n(n+1)}{2}.$$
\end{proof}

\vskip .1in

\subsubsection{The case $r=2n$.} If $r=2n$, for $k=n,\ldots, 2n$ we have
$$\mathcal M(k)=\left(\begin{array}{cccccccc}
m_{n+1}&m_n&\cdots&m_1&m_0&0&\cdots&0\\
m_{n+2}&m_{n+1}&\cdots&m_2&m_1&m_0&\cdots&0\\
\vdots&\vdots& &\vdots&\vdots&\vdots& &\vdots\\
m_{2n}&m_{2n-1}&\cdots&m_n&m_{n-1}&m_{n-2}&\cdots&m_{2n-k}
\end{array}\right),$$ is an $n\times(k+1)$ matrix. Now, $m_{\ell}={{2n+1}\choose{\ell}},\ell=0,\ldots,2n$.

We have $$K=\oplus_{k=n}^{2n}Ker(\mathcal M(k)).$$

\begin{lem} We have $$\dim K=\frac{(n+1)(n+2)}{2}.$$
\end{lem}
\begin{proof} Consider the left-most $n\times n$ matrix block of the matrix $\mathcal M(k)$: $$\mathcal N=\left(\begin{array}{cccc}
{{2n+1}\choose{n+1}}&{{2n+1}\choose{n}}&\cdots&{{2n+1}\choose{2}}\\
{{2n+1}\choose{n+2}}&{{2n+1}\choose{n+1}}&\cdots&{{2n+1}\choose{3}}\\
\vdots&\vdots&\ddots&\vdots\\
{{2n+1}\choose{2n}}&{{2n+1}\choose{2n-1}}&\cdots&{{2n+1}\choose{n+1}}
\end{array}\right).$$

As in the previous case, the determinant of this block is the dimension of the nontrivial representation $S^{\{\lambda_1,\ldots,\lambda_n\}}V, \lambda_i=n+1,$ of $SL(V)$, where $V$ is a vector space of dimension $2n+1$. Therefore, for all $k=n,\ldots,2n$, $rank(\mathcal M(k))=n$.

We have $$\dim K=\sum_{k=n}^{2n}\dim Ker(\mathcal M(k))=\sum_{k=n}^{2n}(k+1-rank(\mathcal M(k)))=\frac{(n+1)(n+2)}{2}.$$
\end{proof}

\vskip .1in

\subsection{Roth's equation and LU-decompositions} In this subsection we will discover how our problem relates to LU-decompositions of matrices and to solving certain types of matrix equations. Again, we divide the analysis into the two cases: $r=2n-1$ and $r=2n$, and we use the results in the previous two subsections: we have the $n\times n$ block denoted with $\mathcal N$ that is invertible.

\subsubsection{The case $r=2n-1$.} For each $k=n,\ldots,2n-1$ we have $$\mathcal M(k)\cdot \left(\begin{array}{c}q_{0,k}\\\vdots\\q_{n-1,k-n+1}\\ \hline q_{n,k-n}\\\vdots\\q_{k,0}\end{array}\right)=\left(\begin{array}{c}0\\\vdots\\0\end{array}\right).$$

Multiplying to the left the above equation by $\mathcal N^{-1}$, putting everything together, we obtain $$\mathcal Q = -\mathcal N^{-1}\cdot \left(\begin{array}{cccc}
m_0&0&\cdots&0\\
m_1&m_0&\cdots&0\\
\vdots&\vdots& &\vdots\\
m_{n-1}&m_{n-2}&\cdots&m_{0}
\end{array}\right)\cdot \left(\begin{array}{cccc}
q_{n,0}&q_{n,1}&\cdots&q_{n,n-1}\\
0&q_{n+1,0}&\cdots&q_{n+1,n-2}\\
\vdots&\vdots& &\vdots\\
0&0&\cdots&q_{2n-1,0}
\end{array}\right).$$

Observe that the lower triangular matrix above is exactly the matrix $\mathcal D$. Also denote with $\tilde{\mathcal Q}$, the upper triangular matrix we see above. The entries of this matrix consist of all the $q_{a,b}$ not occurring in $\mathcal Q$.

We obtain $$\mathcal Q = -\mathcal N^{-1}\mathcal D\tilde{\mathcal Q}.$$

We have the same result for the parameters $s_{u,v}$: $$\mathcal S = -\mathcal N^{-1}\mathcal D\tilde{\mathcal S}.$$

We had that $\mathcal D\mathcal S=\mathcal J\mathcal Q^{T}\mathcal D^{T}\mathcal J$. Denoting with $\mathcal E=\mathcal D\mathcal N^{-1}\mathcal D$, we have $$\mathcal E\tilde{\mathcal S}=\mathcal J(\tilde{\mathcal Q})^T\mathcal E^T\mathcal J,$$ and since $\mathcal J=\mathcal J^{-1}=\mathcal J^T$, denoting with $\mathcal U=\mathcal J\mathcal E$, we obtain $$\mathcal U\tilde{\mathcal S}=(\mathcal U\tilde{\mathcal Q})^T.$$

Our original problem translates into the following question regarding the solution of a certain type of matrix equation:

\begin{prop} We have $\dim K(2n-1)=n$ if and only if the matrix equation $$\mathcal UX-Y^T\mathcal U^T=C,$$ has a solution consisting of two upper-triangular matrices for any $n\times n$ matrix $C$.
\end{prop}
\begin{proof} The matrix equation has the desired solution if and only if the $\mathbb R-$linear map $$\phi: \mathbb R^{\frac{n(n+1)}{2}}\oplus \mathbb R^{\frac{n(n+1)}{2}}\rightarrow \mathbb R^{n^2},$$ given by $\phi(\tilde{\mathcal S},\tilde{\mathcal Q})= \mathcal U\tilde{\mathcal S}-\tilde{\mathcal Q}^T\mathcal U^T$, is surjective. \footnote{Here we have used the fact that an upper-triangular $n\times n$ matrix is described by $\frac{n(n+1)}{2}$ parameters.}

But this is equivalent to the dimension of $ker(\phi)\simeq K(2n-1)$ being equal to $n(n+1)-n^2=n$.
\end{proof}

\vskip .1in

\subsubsection{The case $r=2n$.} For this case the same things occur. For each $k=n,\ldots,2n$ we have $$\mathcal M(k)\cdot \left(\begin{array}{c}q_{0,k}\\\vdots\\q_{n-1,k-n+1}\\ \hline q_{n,k-n}\\\vdots\\q_{k,0}\end{array}\right)=\left(\begin{array}{c}0\\\vdots\\0\end{array}\right).$$

Multiplying to the left the above equation by $\mathcal N^{-1}$ and putting everything together, we obtain $$\left(\begin{array}{ccc}
q_{0,n}&\cdots&q_{0,2n}\\
q_{1,n-1}&\cdots&q_{1,2n-1}\\
\vdots& &\vdots\\
q_{n-1,1}&\cdots&q_{n-1,n+1}
\end{array}\right) = -\mathcal N^{-1}\left(\begin{array}{ccccc}
m_1&m_0&0&\cdots&0\\
m_2&m_1&m_0&\cdots&0\\
\vdots&\vdots&\vdots& &\vdots\\
m_n&m_{n-1}&m_{n-2}&\cdots&m_{0}
\end{array}\right)\left(\begin{array}{ccc}
q_{n,0}&\cdots&q_{n,n}\\
0&\cdots&q_{n+1,n-1}\\
\vdots&\ddots&\vdots\\
0&\cdots&q_{2n,0}
\end{array}\right).$$

The matrix to the left of the equality is an $n\times (n+1)$ matrix consisting of the first $n$ rows of the matrix $\mathcal Q$ we saw before. The missing row occurs as the first row in the upper-triangular matrix we see on the right.

To correct this inconvenience, let $\mathcal N'$ be the $(n+1)\times (n+1)$ matrix $$\mathcal N'=\left(\begin{array}{ccc}
0&|&\mathcal N^{-1}\\ \hline
-1&|&0
\end{array}\right),$$ and since $m_0=1$ observe that $$\mathcal Q=-\mathcal N'\left(\begin{array}{ccccc}
m_0&0&0&\cdots&0\\
m_1&m_0&0&\cdots&0\\
m_2&m_1&m_0&\cdots&0\\
\vdots&\vdots&\vdots& &\vdots\\
m_n&m_{n-1}&m_{n-2}&\cdots&m_{0}
\end{array}\right)\left(\begin{array}{ccc}
q_{n,0}&\cdots&q_{n,n}\\
0&\cdots&q_{n+1,n-1}\\
\vdots&\ddots&\vdots\\
0&\cdots&q_{2n,0}
\end{array}\right).$$

Observe that the lower triangular matrix above is exactly the matrix $\mathcal D$. Also denote with $\tilde{\mathcal Q}$, the upper triangular matrix we see on the right side of the equality.

We obtained $$\mathcal Q = -\mathcal N'\mathcal D\tilde{\mathcal Q}.$$

We have the same taking place for the parameters $s_{u,v}$: $$\mathcal S = -\mathcal N'\mathcal D\tilde{\mathcal S}.$$

We had that $\mathcal D\mathcal S=\mathcal J\mathcal Q^{T}\mathcal D^{T}\mathcal J$. Denoting with $\mathcal E=\mathcal D\mathcal N'\mathcal D$, and with $\mathcal U= \mathcal J\mathcal E$ similarly to the case $r=2n-1$, we have $$\mathcal U\tilde{\mathcal S}=(\mathcal U\tilde{\mathcal Q})^T.$$

As for the other case, with a similar proof, we have:

\begin{prop} We have $\dim K(2n)=n+1$ if and only if the matrix equation $$\mathcal UX-Y^T\mathcal U^T=C,$$ has a solution consisting of two upper-triangular matrices for any $(n+1)\times (n+1)$ matrix $C$.
\end{prop}

\vskip .3in

The equations in the two propositions above are a particular case of {\em Roth's equation} (see \cite{r}): $$AX-YB=C.$$ This equation has solutions in $X$ and $Y$ if and only if the matrices $\left(\begin{array}{ccc}
A&|&0\\ \hline
0&|&B
\end{array}\right)$ and $\left(\begin{array}{ccc}
A&|&C\\ \hline
0&|&B
\end{array}\right)$ have the same rank (see \cite{p2}, Theorem 44.3, page 198).
\vskip .1in

Though in our case $A=\mathcal U$ and $B=\mathcal U^T$, and they are invertible matrices, our goal, and challenge, is to find a special type of solution: we need $X$ to be upper-triangular and $Y$ to be lower-triangular. The next lemma presents one instance when this goal is achieved.

First, we say that an invertible matrix $W$ admits an {\em LU-decomposition} if $W$ has a decomposition: $W=VU$, with $V$ a lower-triangular matrix and $U$ an upper-triangular matrix. It is known that for any invertible matrix $W$ there exists a permutation matrix $P$ such that $PW$ has an LU-decomposition.

\begin{lem}\label{lem:triang} Let $C$ be a $p\times p$ matrix, and $W$ be an invertible $p\times p$ matrix that admits a LU-decomposition. Then there exist two upper triangular matrices $X$ and $Y$, such that $$C=WX-Y^TW^T.$$
\end{lem}
\begin{proof} We have the classical known properties of triangular matrices: (1) the transpose of an upper (lower) triangular matrix is a lower (upper) triangular matrix; (2) the inverse of an upper (lower) triangular matrix is an upper (lower) triangular matrix; (3) the product of two upper (lower) triangular matrices is an upper (lower) triangular matrix.

We have $W=VU$. Let $C'=V^{-1}C(V^T)^{-1}$ and write $C'=C'_u-C'_l$, where $C'_u$ is upper triangular and $C'_l$ is lower triangular (this decomposition is not unique).

Let $X=U^{-1}C'_uV^T$ and $Y=U^{-1}(C'_l)^TV^T$. Both $X$ and $Y$ are upper triangular and we have $$WX-Y^TW^T = VUU^{-1}C'_uV^T-VC'_l(U^T)^{-1}U^TV^T=VC'V^T=C.$$ \end{proof}

Based on possibly not enough experimentation, we ask the following question:

\begin{question} Let $W$ be an invertible matrix. If for any square matrix $C$ the equation $C=WX-Y^TW^T$ has a solution consisting of two upper-triangular matrices, is it true that $W$ has an LU-decomposition?
\end{question}

Regardless if the above question has an affirmative answer or not, one has the following classical criterion: an invertible matrix has LU-decomposition if and only if its leading principal minors are nonzero (see \cite{g}, page 35).

\vskip .3in

Now we go back to our problem. First, let us denote with $\bar{\mathcal N}$ the matrix $\mathcal N^{-1}$ (when $r=2n-1$), and also with $\bar{\mathcal N}$ the matrix $\mathcal N'$ (when $r=2n$). Then our invertible matrix that we would like to show has LU-decomposition is $$\mathcal U=\mathcal J\mathcal D\bar{\mathcal N}\mathcal D.$$

\begin{prop}\label{prop:propertiesU} \begin{enumerate}
                                       \item $\mathcal U$ is symmetric.
                                       \item $\mathcal J\mathcal U\mathcal J$ has LU-decomposition.
                                     \end{enumerate}
\end{prop}
\begin{proof} To prove (1), observe that in both cases for $r$, the matrices $\mathcal J\mathcal N$ and $\mathcal J\mathcal D$ are symmetric; this is true from the fact that multiplying a matrix to the left by the exchange matrix $\mathcal J$, we reverse the rows in the matrix. If we multiply to the right, we reverse the columns.

$\mathcal J^2$ is the identity matrix, so $\mathcal N^{-1}\mathcal J$ is also symmetric. Now, writing $$\mathcal U=(\mathcal J\mathcal D)(\bar{\mathcal N}\mathcal J)(\mathcal J\mathcal D),$$ we obtain the result.

\vskip .2in

To prove (2), observe that $\mathcal D\mathcal J=\mathcal J\mathcal D^T$. With this we have $$\mathcal J\mathcal U\mathcal J=\mathcal D(\bar{\mathcal N}\mathcal J)\mathcal D^T.$$ It becomes enough to show that $\bar{\mathcal N}\mathcal J$ has LU-decomposition, since $\mathcal D$ is lower-triangular and hence $\mathcal D^T$ is upper-triangular.

By the way we denoted what $\bar{\mathcal N}$ is, it will be enough to show that $(\mathcal J\mathcal N)^{-1}$ has LU-decomposition, or equivalently, that $\mathcal J\mathcal N$ has upper-lower decomposition (i.e., UL-decomposition).

In terms of minors, we have to show that the minors of $\mathcal N$ starting from the north-east corner and moving down along the anti-diagonal to the south-west corner, are nonzero: $$det(\mathcal N_{\{1\}\times\{n\}})\neq 0, det(\mathcal N_{\{1,2\}\times\{n-1,n\}})\neq 0,\ldots,det(\mathcal N_{\{1,\ldots,n\}\times\{1,\ldots,n\}})\neq 0.$$

Our matrix $\mathcal N$ is a submatrix of the Toeplitz matrix $$A_T=(a_{j-i})_{i,j}$$ associated to the polynomial $$T(x)=(1+x)^{r+1}=a_0+a_1x+\cdots+a_{r+1}x^{r+1}.$$

The polynomial $T(x)$ has all roots real numbers, equal to $-1<0$, and has coefficients $a_i>0$. From a theorem of Aissen-Schoenberg-Whitney from 1952 (see \cite{p}, Theorem 4.5, page 105), this will imply that all the minors of $A_T$, and therefore of $\mathcal N$, are strictly larger than zero.\footnote{This also gives an alternative proof that $det(\mathcal N)\neq 0$, and so $\mathcal N$ is invertible.}

In particular, the minors of our interest are different than zero, and hence the claim.
\end{proof}

Though Proposition \ref{prop:propertiesU} (2) does not resolve our problem, it gives lower-triangular solutions for our equation $$\mathcal UX-Y^T\mathcal U=C.$$ Part (1) of Proposition \ref{prop:propertiesU} allows us to replace $\mathcal U^T$ by $\mathcal U$.

\begin{cor} For any square matrix $C$, the equation $$\mathcal UX-Y^T\mathcal U=C$$ has a solution consisting of two lower-triangular matrices $X$ and $Y$.
\end{cor}
\begin{proof} From Proposition \ref{prop:propertiesU}, for any square matrix $C$, there exist two upper-triangular matrices $X'$ and $Y'$ such that $$(\mathcal J\mathcal U\mathcal J)X'-(Y')^T(\mathcal J\mathcal U\mathcal J)=\mathcal JC\mathcal J.$$

Then $X=\mathcal J X'\mathcal J$ and $Y=\mathcal J Y'\mathcal J$ are lower-triangular solutions of the equation $$\mathcal UX-Y^T\mathcal U=C.$$
\end{proof}

Is the corollary above enough to show that $\dim K(2n-1)=n$ and $\dim K(2n)=n+1$?

\vskip .1in

\textbf{Acknowledgements:} We would like to thank the anonymous referee for the useful corrections, suggestions and comments.

\renewcommand{\baselinestretch}{1.0}
\small\normalsize 

\bibliographystyle{amsalpha}

\end{document}